\documentclass{amsart}
\usepackage{enumerate,color}

\newtheorem{theorem}{Theorem}[section]
\newtheorem{lemma}[theorem]{Lemma}
\newtheorem{proposition}[theorem]{Proposition}
\newtheorem{cor}[theorem]{Corollary}
\newtheorem{conj}[theorem]{Conjecture}
\newcommand{\li}{\operatorname{li}}
\newcommand{\dist}{\operatorname{dist}}
\renewcommand{\Re}{\operatorname{Re}}
\newcommand{\Sclass}{{\mathcal{S}}}
\newcommand{\meas}[1]{\frac{1}{T}\operatorname{meas}\left\{\tau\in [0,T]: #1\right\}}

\begin{document}
\title[Selber's conjecture and joint universality]{Selberg's orthonormality conjecture and joint universality of L-functions}

\author{Yoonbok Lee}
\address{1. Department of Mathematics, 2. Research Institute of Natural Sciences, Incheon National University, Incheon, Korea}
\email{leeyb131@gmail.com, leeyb@inu.ac.kr}

\author{Takashi Nakamura}
\address{Department of Liberal Arts, Faculty of Science and Technology,
Tokyo University of Science,
2641 Yamazaki, Noda-shi, Chiba-ken, 278-8510, Japan}
\email{nakamuratakashi@rs.tus.ac.jp}

\author{{\L}ukasz Pa\'nkowski}
\address{Faculty of Mathematics and Computer Science, Adam Mickiewicz University, Umultowska 87, 61-614 Pozna\'{n}, Poland, and Graduate School of Mathematics, Nagoya University, Nagoya, 464-8602, Japan}
\email{lpan@amu.edu.pl}

\thanks{The second author was partially supported by JSPS Grant no. 24740029. The third author was partially supported by (JSPS) KAKENHI grant no. 26004317 and the grant no. 2013/11/B/ST1/02799 from the National Science Centre.}



\begin{abstract}
In the paper we introduce a new method how to use only an orthonormality relation of coefficients of Dirichlet series defining given L-functions from the Selberg class to prove joint universality.
\end{abstract}

\maketitle

\section{Introduction}

In 1975, Voronin \cite{V} discovered the so-called universality property, which is one of the most remarkable result concerning the value-distribution of $\zeta(s)$. The modern version states that for any continuous non-vanishing function $f(s)$ on a compact set with connected complement $K\subset\{s\in\mathbb{C}:1/2<\Re(s)<1\}$, analytic in the interior of $K$, we have
\[
\forall_{\varepsilon>0}\liminf_{T\to\infty}\meas{\max_{s\in K}|\zeta(s+i\tau)-f(s)|<\varepsilon}>0,
\]
where $\operatorname{meas}\{\cdot\}$ denotes the Lebesgue real measure. 

Voronin's universality theorem has been generalized for many zeta and $L$-functions from number theory. For example, a universality theorem is known for: Dirichlet $L$-functions (Voronin 1975), Dedekind zeta functions (Reich, 1980), Artin $L$-functions (Bauer 2003), $L$-functions associated with newforms (Laurin\v{c}ikas, Matsumoto and Steuding, 2003), and many others. A quite general class of universal $L$-functions with polynomial Euler product was introduced by Steuding in \cite{S}, and recently, his result was generalized by Nagoshi and Steuding in \cite{NS} to all $L$-functions from the Selberg class with coefficients $a(n)$ of Dirichlet series representation satisfying
\begin{equation}\label{eq:PNT}
\lim_{x\to\infty}\frac{1}{\pi(x)}\sum_{p\leq x}|a(p)|^2 = \kappa,
\end{equation}
for some positive constant $\kappa$ depending on $L$; here $\pi(x)$, as usual, counts the number of primes not exceeding $x$. 

Let us recall that the Selberg class $\Sclass$ consists of functions $L(s)$ defined by a Dirichlet series $\sum_{n=1}^\infty a_L(n)n^{-s}$ in the half-plane $\sigma:=\Re(s) > 1$ satisfying the following axioms:
\begin{enumerate}[(i)]
\item {\it Ramanujan hypothesis:} $a_L(n)\ll_\varepsilon n^\varepsilon$ for every $\varepsilon>0$;
\item {\it analytic continuation:} there exists a non-negative integer $m_L$ such that \break $(s-1)^{m_L}L(s)$ is an entire function of finite order;
\item {\it functional equation:} $L(s)$ satisfies the following functional equation
\[
\Lambda(s) = \theta\overline{\Lambda(1-\overline{s})},
\]
where
\[
\Lambda(s):=L(s)Q^s\prod_{j=1}^k\Gamma(\lambda_js+\mu_j),
\]
$|\theta|=1$, $Q,\lambda_j\in\mathbb{R}$, and $\mu_j\in\mathbb{C}$ with $\Re(\mu_j)\geq 0$;
\item {\it Euler product:} for $\sigma>1$ we have
\[
\log L(s) = \sum_p\sum_{k=1}^\infty\frac{b_L(p^k)}{p^{ks}},
\]
where $b_L(p^k)$ are complex numbers satisfying $b_L(p^k)\ll p^{k\theta}$ for some $\theta<1/2$. 
\end{enumerate}

The condition \eqref{eq:PNT} is closely related to the following widely believed Selberg conjecture.
\begin{conj}[Selberg]\label{conj:Selberg}
For any function $1\ne L\in \Sclass$ there is a positive integer $\kappa_L$ such that
\begin{equation}\label{eq:SelbergA}
\sum_{p\leq x}\frac{|a_L(p)|^2}{p} = \kappa_L\log\log x + R(x)
\end{equation}
and, for any primitive functions $L_1, L_2\in\Sclass$, we have 
\begin{equation}\label{eq:SelbergB}
\sum_{p\leq x}\frac{a_{L_1}(p)\overline{a_{L_2}(p)}}{p} = R(x),
\end{equation}
where $R(x)\ll 1$.
 
The last equation can be called the orthonormality relation.
\end{conj}

It should be noted that it is expected that \eqref{eq:SelbergA} with $R(x)\ll 1$ is too weak to prove universality for a single $L$-function, because of lack of a sufficiently good error term (see \cite[the footnote on p. 129]{S}). Therefore, it is natural that \eqref{eq:PNT} is the stronger assumption than Selberg's conjecture. It implies that $R(x)\ll 1$, but to prove \eqref{eq:PNT} we need \eqref{eq:SelbergA} with 
\[
R(x)=C_1+\frac{C_2}{\log x} + O\left(\frac{1}{(\log x)^2}\right).
\]

Moreover, note that almost all known proofs of universality requires existing of the mean-square, which is rather difficult problem in the general setting of Selberg class. For example, the best known result (see \cite{Potter} or \cite[Corollary 6.11]{S}) says that, for $L\in\Sclass$, we have
\[
\lim_{T\to\infty}\frac{1}{2T}\int_{-T}^T |L(\sigma+it)|^2dt = \sum_{n=1}^\infty\frac{|a_L(n)|^2}{n^{2\sigma}}<\infty,\qquad\sigma>\max\left\{\frac{1}{2},1-\frac{1}{d_L}\right\},
\]
where $d_L$ denotes the degree of $L$ defined by $2\sum_{j=1}^k\lambda_j$, where $\lambda_j$'s are given by the functional equation of $L$. Therefore, it is natural that Nagoshi's and Steuding's universality theorem of $L$-function from the Selberg class was proved only in the strip $\{s\in\mathbb{C}:\sigma_{\textup m}(L)<\Re(s)<1\}$, where $\sigma_{\textup{m}}(L)$ denotes the abscissa of the mean-square half-plane for $L$.\pagebreak

Voronin in \cite{V2} (see also \cite[Chapter VII, Theorem 3.2.1]{KV}) proved also the so-called joint universality theorem for Dirichlet $L$-functions associated with pairwise non-equivalent Dirichlet characters. Roughly speaking, he proved that any collection of analytic non-vanishing functions $f_1,f_2,\ldots, f_n$ can be approximated, in the Voronin sense, by the shift $L(s+i\tau;\chi_1), L(s+i\tau;\chi_2),\ldots,L(s+i\tau;\chi_n)$, where $\chi_1,\ldots,\chi_n$ are pairwise non-equivalent Dirichlet characters. Joint universality was also proved for many other zeta and $L$-functions from number theory. However, it is still open problem put forward by Steuding in \cite{S}, whether the collection of $L$-functions from Selberg class is jointly universal under the assumption of Selberg's orthonormality conjecture \eqref{eq:SelbergB}. Obviously, to expect joint universality for at least two functions $L_1$ and $L_2$ we need some kind of their independence, so Selberg's conjecture seems to be the most natural assumption of this kind in the Selberg class. Interesting evidence for the truth of this conjecture was given by Bombieri and Hejhal in \cite{BH}, where they showed the statistical independence of any collection of $L$-functions under a stronger version of Selberg's conjecture. Moreover, it is known that Selberg's conjecture with $R(x)\ll 1$ is not sufficient to prove joint universality. The second author in \cite[Example 7.5]{N} observed that, for non-principle Dirichlet character $\chi$, the Dirichlet $L$-functions $L(s,\chi)$ and $L(s-i,\chi)$ cannot be jointly universal, whereas it is easy to observe that
\[
\sum_{p\leq x}\frac{\chi(p)\overline{\chi(p)p^{-i}}}{p}=\sum_{\substack{p\leq x\\\chi(p)\ne 0}}p^{i-1}\ll 1,
\]
so Selberg's conjecture with $R(x)\ll 1$ holds.

The main purpose of this paper is to introduce a new method how to use only orthonormality to prove joint universality of $L$-functions with Euler product. In order to illustrate this idea we prove a general joint universality theorem for any collection of $L$-functions $L_1,\ldots,L_m$ from the Selberg class satisfying some stronger analogue of Selberg's conjecture, namely
\begin{equation}\label{eq:SelbergSA}
\sum_{p\leq x}|a_{L_k}(p)|^2 = \sum_{j=1}^{2m+1}\frac{c^{(k)}_jx}{(\log x)^j} + O\left(\frac{x}{(\log x)^{2m+2}}\right)\qquad(1\leq k\leq m),
\end{equation}
and
\begin{equation}\label{eq:SelbergSB}
\sum_{p\leq x}a_{L_k}(p)\overline{a_{L_l}(p)} = \sum_{j=2}^{2m+1}\frac{c^{(k,l)}_j x}{(\log x)^j} + O\left(\frac{x}{(\log x)^{2m+2}}\right)\qquad (1\leq k\ne l\leq m),
\end{equation}
where $c^{(k)}_j, c^{(k,l)}_j$ are some constants and $c^{(k)}_1>0$.
It is easy to observe, by partial summation, that it is equivalent to the Selberg's conjecture \eqref{eq:SelbergA} and \eqref{eq:SelbergB}, where
\[
R(x) = \sum_{j=0}^{2m+2}\frac{c_j}{(\log x)^j} + O\left(\frac{1}{(\log x)^{2m+3}}\right)
\]
for suitable $c_j$ depending on given $L$-functions. 

Although the above formulas are obviously stronger than the original Selberg's conjecture, it is quite likely that they are fulfilled by all $L$-functions. We refer to Section \ref{sec:Ex} for a detailed discussion of this matter, where several unconditional joint universality theorems for automorphic $L$-functions are deduced from our method. Here we only mention that the evidence for the truth of this conjecture is the fact that there is a grand hypothesis that each $L$-function from Selberg class can be defined as a suitable automorphic $L$-function and, so far, all automorphic $L$-functions satisfying Selberg's conjecture fulfill in fact \eqref{eq:SelbergSA} and \eqref{eq:SelbergSB}.

\begin{theorem}\label{th:main}
Let $L_1,\ldots,L_m$ be elements of $\Sclass$, $K_1,\ldots,K_m\subset \{s\in\mathbb{C}:\break\max_{j=1,2,\ldots,m}{\sigma_{\textup{m}}(L_j)}<\Re s< 1\}$ be compact sets with connected complements and $g_j$, $j=1,\ldots,m$ be continuous non-vanishing function on $K_j$, and analytic in the interior of $K_j$. Then, if \eqref{eq:SelbergSA} and \eqref{eq:SelbergSB} hold, we have, for every $\varepsilon>0$, that
\[
\liminf_{T\to\infty}\meas{\max_{j=1,\ldots,m}\max_{s\in K_j}|L_j(s+i\tau)-g_j(s)|<\varepsilon}>0.
\]
\end{theorem}

Noteworthy is the fact that most of proofs of universality rely on periodicity and orthonormality property of coefficients of $L$-functions. Recently, Mishou in \cite{M} invented a new approach to prove joint universality without periodicity, which works for a pair of $L$-functions with real coefficients under the assumption of some analogue of \eqref{eq:SelbergSA} and \eqref{eq:SelbergSB}. The purpose of this paper is to introduce another new approach how to use only orthonormality relation to prove joint universality for any collection of $L$-functions with complex coefficients. This method can be easily generalized to other zeta and $L$-functions, which joint universality property relies on some independence of coefficients of Dirichlet series representation. For example, in \cite{LNP} the authors proved joint universality for a collection of Lerch zeta functions $L(s;\alpha,\lambda_j) = \sum_{n=0}^\infty \frac{\exp(2\pi i\lambda_j)}{(n+\alpha)^s}$, $j=1,2,\ldots,m$, associated to transcendental $\alpha\in (0,1]$ and distinct $\lambda_j$'s with $\lambda_j\in (0,1]$.

As standard consequences of universality, one can easily prove the following corollaries. For the proofs we refer, for example, to \cite[Section 8]{M}, where Mishou showed similar results for a pair of $L$-functions. However, the modifications needed are straightforward and can be left to the reader (see \cite[Section 10]{S}).

\begin{cor}
Let $m\geq 2$, $0\ne a_1,\ldots,a_m\in\mathbb{C}$ and $L_1,\ldots,L_m\in\Sclass$ satisfy \eqref{eq:SelbergSA} and \eqref{eq:SelbergSB}. Then the function
\[
L(s) = \sum_{j=1}^m a_jL_j(s)
\]
is strongly universal in the strip $\max_{1\leq j\leq m}\sigma_{\textup m}(L_j)=:\sigma_L<\sigma<1$, which means that Theorem \ref{th:main} holds also for functions $g_j$ having zeros on $K_j$.

Moreover, the function $L(s)$ has infinitely many zeros in the strip $\sigma_L<\sigma<1$, namely for any $\sigma_1,\sigma_2$ with $\sigma_L<\sigma_1<\sigma_2<1$ and sufficiently large $T$ there exist $\gg T$ zeros $\rho=\beta+i\gamma$ of $L(s)$ in the rectangle $\sigma_1\leq \beta\leq\sigma_2$, $0\leq \gamma\leq T$.
\end{cor}

\begin{cor}
Let $N\in\mathbb{N}$, $L_1,\ldots,L_m\in\Sclass$ satisfy \eqref{eq:SelbergSA} and \eqref{eq:SelbergSB} and $\sigma_0$ be a real number satisfying $\max_{1\leq j\leq m}\sigma_{\textup m}(L_j)<\sigma_0<1$. Then the set
\[
\left\{\left(L_1(\sigma_0+it),\ldots,L_m(\sigma_0+it),\ldots,L^{(N-1)}_1(\sigma_0+it),\ldots,L^{(N-1)}_m(\sigma_0+it)\right):t\in\mathbb{R}\right\}
\]
is dense in $\mathbb{C}^{mN}$.
\end{cor}

\begin{cor}
Let $N\in\mathbb{N}$ and $L_1,\ldots,L_m\in\Sclass$ satisfy \eqref{eq:SelbergSA} and \eqref{eq:SelbergSB}. If continuous functions $f_l:\mathbb{C}^{mN}\to\mathbb{C}$, $l=0,1,\ldots,L$ satisfy
\[
\sum_{l=0}^L s^lf_l\left(L_1(s),\ldots,L_m(s),\ldots,L^{(N-1)}_1(s),\ldots,L^{(N-1)}_m(s)\right)\equiv 0
\]
for all $s\in\mathbb{C}$, then $f_l\equiv 0$ for all $0\leq l\leq L$.
\end{cor}


\section{A denseness lemma}

Let us fix $L$-functions $L_1,\ldots,L_m\in\Sclass$ and compact sets $K_1,\ldots,K_m\subset \{s\in\mathbb{C}:\max_{j=1,2,\ldots,m}{\sigma_{\textup{m}}(L_j)}<\Re s< 1\}$. Take $\sigma_1>\max_{j=1,2,\ldots,m}\{\sigma_{\textup{m}}(L_j)\}$ and $\sigma_2<1$ such that $K_j$, $j=1,2,\ldots,m$, are the subset of the strip $D:=\{s\in\mathbb{C}:\sigma_1<\Re s< \sigma_2\}$ and denote the space of analytic functions on $D$ equipped with the topology of uniform convergence on compacta by $H(D)$. 

Then, the main purpose of this section  is to prove the so-called denseness lemma in the space $H(D)^m$, which plays a crucial role in the proof of universality and says that any collection of analytic functions from $H(D)^m$ can be approximated by given $L$-functions $L_1,\ldots,L_m$ twisted by certain sequence of complex numbers with absolute value $1$.

In order to show it, let $\gamma:=\{s\in\mathbb{C}: |s|=1\}$ and $\Omega:=\prod_{p}\gamma_p$ be an infinite-dimensional torus with product topology and pointwise multiplication, where $\gamma_p=\gamma$ for each prime $p$. It is well known that $\Omega$ is a compact topological abelian group, so there is a normalized Haar measure $m_H$ on $(\Omega,\mathcal{B}(\Omega))$, where $\mathcal{B}(\Omega)$ denotes the class of Borel sets of $\Omega$. 

Let $\omega(p)$ denote the projection of $\omega\in\Omega$ to the coordinate space $\gamma_p$ and $\omega:\mathbb{N}\to\mathbb{C}$ be a unimodular completely multiplicative extansion of $\omega$. Then for any $L\in\Sclass$ defined for $\sigma>1$ by the series  $\sum_{n=1}^\infty a_L(n)n^{-s}$ we put
\[
L(s,\omega) = \sum_{n=1}^\infty\frac{a_L(n)\omega(n)}{n^s},\qquad s\in D.
\]
It turned out (see for example \cite[Lemma 4.1]{S}) that $L(s,\omega)$ is a random element on the probabilistic space $(\Omega,\mathcal{B}(\Omega),m_H)$ and  
for almost all $\omega\in\Omega$ we have (see \cite[Eq.~(3.17)]{NS})
\[
\log L(s,\omega) = \sum_p\sum_{k=1}^{\infty}\frac{b_L(p^k)\omega(p)^k}{p^{ks}},\qquad(s\in D).
\]

Thus, for $L_j\in\Sclass$, $j=1,2,\ldots,m$, let us put
\[
g_{p,j}(s,\omega(p)) = \sum_{k=1}^\infty\frac{b_{L_j}(p^k)\omega(p)^k}{p^{ks}},\qquad \omega\in\Omega
\]
and 
\[
\underline{g}_p(s,\omega(p)) = (g_{p,1}(s,\omega(p)),\ldots, g_{p,m}(s,\omega(p))).
\]

Therefore, the main result of this section is the following proposition, which strongly relies on Selberg's conjecture.
\begin{proposition}\label{prop:dense}
If we assume the truth of \eqref{eq:SelbergSA} and $\eqref{eq:SelbergSB}$, then the set of convergent series
\[
\left\{\sum_p\underline{g}_p(s,\omega(p)):\omega\in\Omega\right\}
\]
is dense in the space $H(D)^m$.
\end{proposition}

Let $U$ be a bounded simply connected smooth Jordan domain satisfying $\overline{U}\subset D$ and $K_j\subset U$ for every $j=1,2,\ldots,m$. Let $L^2(U)$ be the complex Hilbert space of all square integrable complex functions on $U$ with the inner product
\[
\langle f,g\rangle = \iint_U f(s)\overline{g(s)}d\sigma dt.
\]
Define the Bergman space $H_1$ as the closure of $H(D)$ in $L^2(U)$. Then $H_1^m$ is the complex Hilbert space with the inner product given, for $\underline{f}=(f_1,\ldots,f_m)$ and $\underline{g}=(g_1,\ldots,g_m)$, by
\[
\langle \underline{f},\underline{g}\rangle = \sum_{j=1}^m\iint_U f_j(s)\overline{g_j(s)}d\sigma dt.
\]

Now, define
\[
\underline{h}_p(s) = (h_{p,1}(s),\ldots,h_{p,m}(s)) := \left(\frac{a_{L_1}(p)}{p^s},\ldots,\frac{a_{L_m}(p)}{p^s}\right).
\]
Then, by the fact that $b_L (p^k) \ll p^{k\theta}$ for some $\theta < 1/2$, one can easily prove that 
\[
\sum_p r_{p,j}(s,\omega):=\sum_p \Big(g_{p,j}(s,\omega(p)) - \omega(p)h_{p,j}(s)\Big),\qquad(j=1,\ldots,n,\ |\omega(p)|=1)
\]
is absolutely convergent on $\overline{U}$.

Hence, in order to prove Proposition \ref{prop:dense} it suffices to prove that the set of all convergent series
\begin{equation}\label{eq:DenseSet}
\left\{\sum_{p>v}\omega(p)\underline{h}_p(s):\omega\in\Omega\right\}
\end{equation}
is dense in $H_1^m$ for an arbitrary given $v>0$. Indeed, let $v$ be a sufficiently large number such that
\[
\sum_{j=1}^m\max_{s\in\overline{U}}\left|\sum_{p>v}r_{p,j}(s,\omega)\right|<\frac{\varepsilon}{2}\qquad\text{for all $\omega\in\Omega$}.
\]
The fact that for every $f\in H_1$ with the norm $||f||$ and $s\in U$ we have $|f(s)|<\frac{||f||}{\sqrt{\pi}\dist(s,\partial {U})}$ (see for example \cite[Chapter I, Section 1, Lemma 1]{G}) clearly implies that the approximation in respect to the norm $\Vert\cdot\Vert$ in $H_1$ gives the uniform approximation on every compact subset $K$ of $U$. Hence, from the fact that the set \eqref{eq:DenseSet} is dense in $H_1^m$, we obtain that, for every $\underline{f}=(f_1,\ldots,f_m)\in H(D)^m$, there exists a sequence $\omega'(p)$ such that
\[
\max_{1\leq j\leq m}\max_{s\in K_j}\left|\sum_{p>v}\omega'(p)h_{p.j}(s) - f_j(s)+\sum_{p\leq v}g_{p,j}(s,1)\right|< \frac{\varepsilon}{2}.
\] 
Therefore, putting
\[
\omega(p) = \begin{cases}1&\text{if $p\leq v$},\\
\omega'(p)&\text{if $p>v$}\end{cases}
\]
gives that
\begin{equation}\label{eq:Dense}
\max_{1\leq j\leq m}\max_{s\in K_j}\left|\sum_{p}g_{p.j}(s,\omega(p)) - f_j(s)\right|< \varepsilon.
\end{equation}

In order to prove that the set \eqref{eq:DenseSet} is dense in $H_1^m$ we shall use the following lemma for the sequence $\underline{h}_p(s)$ and the Hilbert space $H_1^m$.
\begin{lemma}\label{lem:complexHilbert}
Let $H$ be a complex Hilbert space. Assume that a sequence $u_n\in H$, $n\in\mathbb{N}$, is such that
\begin{enumerate}[\upshape (i)]
\item the series $\sum_n ||u_n||^2< \infty$;
\item for any element $0\ne e\in H$ the series $\sum_n |\langle u_n,e\rangle|$ is divergent.
\end{enumerate}
Then the set of convergent series
\[
\left\{\sum_n a_nu_n\in H:|a_n|=1\right\}
\]
is dense in $H$.
\end{lemma}
\begin{proof}
This is \cite[Theorem 5.4]{S}.
\end{proof}

Since $\Re s>\sigma_1>1/2$ for all $s\in U$, one can easily show that
\[
\sum_p || \underline{h}_p (s) ||^2  < \infty
\]
and the condition (i) holds. 

Now let $\underline{g}=(g_1,\ldots,g_m)\in H_1^m$ be a non-zero element. Then 
\[
\langle \underline{h}_p(s),\underline{g}(s)\rangle = \sum_{j=1}^m a_{L_j}(p)\Delta_j(\log p),
\]
where $\Delta_j(z) = \iint_U e^{-sz}\overline{g_j(s)}d\sigma dt$. 
Then, in order to complete the proof of Propositon \ref{prop:dense}  it suffices to prove the following lemma.

\begin{lemma}\label{lem:main}
Let $\underline{g}(s) = (g_1(s),\ldots,g_m(s))\in H_1^m$ be a non-zero element and $\Delta_j(z) =\iint_U e^{-sz}\overline{g_j(s)}d\sigma dt$. Then, assuming Selberg's conjecture\eqref{eq:SelbergSA} and \eqref{eq:SelbergSB} for $L_1,\ldots,L_m\in\Sclass$ gives that the series
\[
\sum_p \left|a_{L_1}(p)\Delta_1(\log p) + \cdots + a_{L_m}(p)\Delta_m(\log p)\right|
\]
is divergent.
\end{lemma}

Before we prove the above lemma, we need to obtain good estimation for $\Delta(\log p)=\iint_U p^{-s}\overline{g(s)}d\sigma dt$, where $g(s)$ is a given non-zero element of $H_1$. In order to prove it we use Markov's inequality.

\begin{lemma}[Markov's inequality]
Suppose that $P(t)$ is a polynomial of degree $n$ with real coefficients, which satisfies 
\[
\max_{t\in [-1,1]}|P(t)|\leq 1.
\]
Then for every $t\in [-1,1]$ we have
\[
|P'(t)|\leq n^2.
\]
\end{lemma}
\begin{proof}
For a proof see for example \cite{Achieser}.
\end{proof}

\begin{cor}\label{cor:Markov}
Let $P(s)$ be polynomial of degree $n$ with complex coefficients. Then for every $a,b$ with $a<b$ and every real $t\in [a,b]$ we have
\[
|P'(t)|\leq \frac{2n^2}{b-a}\max_{t\in [a,b]}|P(t)|.
\]
\end{cor}
\begin{proof}
Let $t_0\in [a,b]$ be such that $|P(t_0)| = \max_{t\in [a,b]}|P(t)|$. Then let us define
\[
P_1(t) = \frac{P\left(\frac{b-a}{2}t+\frac{a+b}{2}\right)}{|P(t_0)|}.
\]
Now, let us take an arbitrary $t\in[-1,1]$ and let $c\in\mathbb{C}$ with $|c|=1$ be such that $cP'_1(t)$ be real. Then applying Markov's inequality for $P_2(t):=\Re (cP_1(t))$ gives
\[
|P'_1(t)|=|cP'_1(t)| = |P'_2(t)|\leq n^2,
\]
so
\[
\max_{t\in[-1,1]}|P'_1(t)|\leq n^2.
\]
On the other hand, we can easily observe that
\[
\max_{t\in[-1,1]}|P'_1(t)| = \frac{b-a}{2|P(t_0)|}\max_{t\in [a,b]}|P'(t)|
\]
and the proof is complete.
\end{proof}

\begin{lemma}\label{lem:3}
Let $U\subset\mathbb{C}$ be open and bounded set and $g$ be Lebesgue square integrable function on $U$. For $z\in\mathbb{C}$ we put
\[
\Delta(z)=\iint_U e^{-sz}\overline{g(s)}d\sigma dt.
\]
Then for every $A>0$ and every interval $I=[x,x+\frac{B}{x^M}]\subset[x,x+1]$ with $B>0$, $M\geq 0$, $x>2$ there exist an interval $I'\subset I$ of length $|I'|\geq \frac{B'}{x^{M+2}}$ with $B':=B'(B,A)>0$ and $x_0\in I'$ such that for all $\xi\in I'$ we have
\[
\frac{1}{2}|\Delta(x)|+O\left(e^{-Ax}\right)\leq \frac{1}{2}|\Delta(x_0)|+O\left(e^{-Ax}\right)\leq |\Delta(\xi)|\leq |\Delta(x_0)|+O\left(e^{-Ax}\right).
\]
Moreover, for every $\xi\in I$ we have
\[
|\Delta'(\xi)|\ll x^{M+2}|\Delta(x_0)|+O(x^{M+2}e^{-Ax}).
\]
\end{lemma}
\begin{proof}
Let $c_0>0$, $K=[c_0x]$ and $C>0$ be such that $\max_{s\in \overline{U}}|s|\leq C$. Then, for every $\xi\in [x,x+1]$, by Stirling's formula we get
\begin{align*}
e^{-s\xi} &= \sum_{l=0}^{K}\frac{(-s\xi)^l}{l!} + O\left(\sum_{l=0}^{\infty}\frac{(xC)^{l+K+1}}{(l+K+1)!}\right)\\
&= \sum_{l=0}^{K}\frac{(-s\xi)^l}{l!} + O\left(\frac{(xC)^{K+1}}{(K+1)!}\sum_{l=0}^{\infty}\frac{(xC)^{l}(K+1)!}{(l+K+1)!}\right)\\
&= \sum_{l=0}^{K}\frac{(-s\xi)^l}{l!} + O\left(e^{xC}\exp\left(-(K+1)\log\left(\frac{K+1}{xC}\right)\right)\right)\\
&= \sum_{l=0}^{K}\frac{(-s\xi)^l}{l!} + O\left(e^{xC}\exp\left(-c_0x\log\left(\frac{c_0}{C}\right)\right)\right).
\end{align*}
Similarly,
\begin{align*}
(-s)e^{-s\xi} &= \sum_{l=0}^{K-1}\frac{(-s)^{l+1}\xi^l}{l!} +  O\left(e^{xC}\exp\left(-(c_0x-1)\log\left(\frac{c_0}{2C}\right)\right)\right).
\end{align*}
Hence, for every $A>0$ there exists sufficiently large $c_0=c_0(A,C)$ such that
\[
e^{-s\xi} = \sum_{l=0}^{K}\frac{(-s\xi)^l}{l!} +  O\left(e^{-x(A+C)}\right)
\]
and
\[
(-s)e^{-s\xi} = \sum_{l=0}^{K-1}\frac{(-s)^{l+1}\xi^l}{l!} +  O\left(e^{-x(A+C)}\right).
\]
for every $\xi\in [x,x+1]$.

Therefore, for $\xi\in[x,x+1]$ we have
\begin{equation}\label{eq:PolyApprox}
\Delta(\xi)=P(\xi)+O(e^{-Ax})\qquad\text{and}\qquad \Delta'(\xi) = P'(\xi) + O(e^{-Ax})
\end{equation}
where $P(\xi)=\sum_{l=0}^K\frac{\xi^l}{l!}\iint_U (-s)^l\overline{g(s)}d\sigma dt$ is a polynomial of degree $\ll x$. 

Let $x_0\in I$ be such that $|P(x_0)| = \max_{\xi\in I}|P(\xi)|$. Then by Corollary \ref{cor:Markov} we get
\begin{equation*}
\max_{\xi\in I}|P'(\xi)|\ll x^{M+2}|P(x_0)|
\end{equation*}
and hence 
\[
|\Delta'(\xi)| = |P'(\xi)|+O(e^{-Ax})\ll x^{M+2}|\Delta(x_0)|+O(x^{M+2}e^{-Ax}).
\]

Therefore, for $\xi\in I$ satisfying $|\xi-x_0|\leq \frac{B'}{x_0^{M+2}}$ with sufficiently small $B'>0$ we have
\begin{equation}\label{eq:meanValue}
\begin{split}
|P(x_0)|-|P(\xi)|&\leq |P(\xi)-P(x_0)|\leq |\xi-x_0|\max_{\xi\in I}|P'(\xi)|\leq \frac{1}{2}|P(x_0)|.
\end{split}
\end{equation}
Therefore, for $\xi\in I':=I\cap \Big[x_0-\frac{B'}{x_0^{M+2}},x_0+\frac{B'}{x_0^{M+2}}\Big]$ it holds
\begin{align*}
\frac{1}{2}|P(x)|\leq \frac{1}{2}|P(x_0)|\leq |P(\xi)|\leq |P(x_0)|,
\end{align*}
and hence, by \eqref{eq:PolyApprox}, the proof is complete.
\end{proof}

\begin{cor}\label{cor:Delta}
Let $U\subset\mathbb{C}$ be open and bounded and $g_j$, $j=1,2,\ldots,m$, be Lebesgue square integrable functions on $U$. For $z\in\mathbb{C}$ we put
\[
\Delta_j(z)=\iint_U e^{-sz}\overline{g_j(s)}d\sigma dt.
\]
Then for every $A>0$ and every $x>1$ there exist $B_1>\cdots>B_m>0$, $x^{(0)}_0=x, x^{(1)}_0,\ldots, x^{(m)}_0$ and intervals $I_j\subset[x,x+1]$ of length $|I_j|\geq \frac{B_j}{x^{2j}}$ such that $x^{(j)}_0\in I_j$, $I_{j+1}\subset I_j$, and for all $\xi\in I_{j}$ we have
\begin{multline*}\label{eq:DeltaEst}
\frac{1}{2}|\Delta_j(x^{(j-1)}_0)|+O\left(e^{-Ax}\right)\leq \frac{1}{2}|\Delta_j(x^{(j)}_0)|+O\left(e^{-Ax}\right)\\
\leq |\Delta_j(\xi)|\leq |\Delta_j(x^{(j)}_0)|+O\left(e^{-Ax}\right).
\end{multline*}
Moreover, for every $t\in I_j$ we have
\[
|\Delta_j'(\xi)|\ll x^{2j}|\Delta_j(x^{(j)}_0)|+O(x^{2j}e^{-Ax}).
\]
\end{cor}
\begin{proof}
Firstly, let us apply the last lemma for $\Delta_1(z)$ and the interval $I_0:=[x,x+1]$. Then there is an interval $I_1\subset I_0$ of length $|I_1|\geq \frac{B_1}{x^{2}}$ and $x^{(1)}_0\in I_1$ such that for $\xi\in I_1$ we have
\[
\frac{1}{2}|\Delta_1(x)|+O\left(e^{-Ax}\right)\leq \frac{1}{2}|\Delta_1(x^{(1)}_0)|+O\left(e^{-Ax}\right)\leq |\Delta_1(\xi)|\leq |\Delta_1(x^{(1)}_0)|+O\left(e^{-Ax}\right)
\]
and
\[
|\Delta_1'(\xi)|\ll x^{2}|\Delta_1(x^{(1)}_0)|+O(x^{2}e^{-Ax}).
\]

Next, we apply again the last lemma for $\Delta_2(z)$ and the interval $I_1=[x',x'+\frac{B_1}{x^{2}}]\subset[x,x+1]$. Thus there is an interval $I_2\subset I_1$ of length $|I_2|\geq \frac{B'_2}{x'^{4}}\geq \frac{B_2}{x^{4}}$ and $x^{(2)}_0\in I_2$ such that
\[
\frac{1}{2}|\Delta_2(x^{(1)}_0)|+O\left(e^{-Ax}\right)\leq \frac{1}{2}|\Delta_2(x^{(2)}_0)|+O\left(e^{-Ax}\right)\leq |\Delta_2(\xi)|\leq |\Delta_2(x^{(2)}_0)|+O\left(e^{-Ax}\right)
\]
and
\[
|\Delta_2'(\xi)|\ll x^{4}|\Delta_2(x^{(2)}_0)|+O(x^{4}e^{-Ax}).
\]
Next, repeating the application of the last lemma for each function $\Delta_j$, $3\leq j\leq m$, completes the proof.
\end{proof}

\begin{proof}[Proof of Lemma \ref{lem:main}]
Without loss of generality we can assume that $g_1$ is a non-zero element, since the fact that $\underline{g}\ne 0$ implies that at least one of $g_j$'s is a non-zero element.

Obviously, $\Delta_1(z)\ll e^{C|z|}$ for some positive constant $C$ depending on $U$. Let us recall that for all $s\in U$ we have $1/2<\sigma_1<\Re  s< \sigma_2<1$. Then for sufficiently small $\eta=\eta(U)>0$ and for all complex $z$ with $|\arg(-z)|\leq \eta$ we have
\[
|e^{\sigma_2 z}\Delta_1(z)|\ll 1.
\]
Moreover, $\Delta_1\not\equiv 0$, since otherwise for every positive integer $k$ we have $0 = \Delta_1^{(k)}(0) = \iint_U (-s)^k \overline{g_1(s)}d\sigma dt$, which means that $g_1$ is orthogonal to all polynomials in $L_2(U)$ and we get contradiction to the fact that $g_1$ is a non-zero element and the linear space of polynomials is dense in the Bergman space $H_1$ (see for example \cite[Theorem 7.2.2]{QQ}). Hence, by \cite[Lemma 3]{KK}, which proof based on the Phragm\'en-Lindel\"of theorem, there is a real sequence $x_k$ tending to $\infty$ such that
\[
|\Delta_1(x_k)|\gg e^{-\sigma_2 x_k}.
\]

Let us fix $k$ and put $x=x_k$. Hence, using Corollary \ref{cor:Delta}, for every $A>0$ and $x=x_k$ there exist $B_1>\cdots>B_m>0$, $x^{(0)}_0=x, x^{(1)}_0,\ldots, x^{(m)}_0$ and intervals $I_j\subset[x,x+1]$ of length $|I_j|\geq \frac{B_j}{x^{2j}}$ such that $x^{(j)}_0\in I_j$, $I_{j+1}\subset I_j$, and for all $\xi\in I_{j}$ we have
\begin{multline}\label{eq:Estimate}
\frac{1}{2}|\Delta_j(x^{(j-1)}_0)|+O\left(e^{-Ax}\right)\leq \frac{1}{2}|\Delta_j(x^{(j)}_0)|+O\left(e^{-Ax}\right)\\
\leq |\Delta_j(\xi)|\leq |\Delta_j(x^{(j)}_0)|+O\left(e^{-Ax}\right)
\end{multline}
and
\begin{equation}\label{eq:DerUpper}
|\Delta_j'(\xi)|\ll x^{2j}|\Delta_j(x^{(j)}_0)|+O(x^{2j}e^{-Ax}).
\end{equation}

Now let $I:=I_m=\Big[x',x'+\frac{B_m}{x'^{2m}}\Big]\subset[x,x+1]$. Since $I\subset I_j$ for every $j=1,2,\ldots,m$, the above inequalities hold also for all $\xi\in I$.

In particular, since $x^{(0)}_0 = x$, for $\xi\in I$ we have
\[
|\Delta_1(\xi)|\geq \frac{1}{2}|\Delta_1(x^{(0)}_0)|\gg e^{-\sigma_2 x}.
\]

Moreover, for every $j=1,2,\ldots,m$ we have
\[
|\Delta_j(\xi)|\ll e^{-\sigma_1 x}\qquad (\xi\in [x,x+1]).
\]

Now, let ${\sum_p}^*$ denote the sum over primes $p\in \Big[e^{x'}, e^{x'+\frac{B_m}{x'^{2m}}}\Big]$. Then for these $p$ we have $\log p\in I$.

It is easy to notice that
\begin{align*}
S(x):&={\sum_p}^* \left |a_{L_1}(p)\Delta_1(\log p)+\cdots+a_{L_m}(p)\Delta_m(\log p)\right|^2\\
&=\sum_{j=1}^m{\sum_p}^* |a_{L_j}(p)|^2|\Delta_j(\log p)|^2 \\
&\quad +\sum_{1\leq k\ne l\leq m}{\sum_p}^* a_{L_k}(p)\overline{a_{L_l}(p)}\Delta_k(\log p)\overline{\Delta_l(\log p)}.
\end{align*}

Using \eqref{eq:SelbergSB} it is easy to prove that for any $1\leq k\ne l<m$ we have
\begin{align*}
\phi_{k,l}(u):\!&=\sum_{p\leq u}a_{L_k}(p)\overline{a_{L_l}(p)} =\sum_{j=2}^{2m+1}\frac{c^{(k,l)}_j u}{(\log u)^j} + O\left(\frac{u}{(\log u)^{2m+2}}\right).
\end{align*}
For $\log u\in I$, by \eqref{eq:DerUpper}, we get
\[
\frac{d}{du}\Delta_j(\log u) = \frac{1}{u}\Delta'_j(\log u)\ll\frac{x^{2m}}{u}|\Delta_j(x^{(j)}_0)|+O(x^{2m}e^{-Ax})
\]
and, since $\overline{\Delta_j(\log u)}=\overline{\langle u^{-s},g_j(s)\rangle} = \langle u^{-\overline{s}},\overline{g_j(s)}\rangle$, we have
\begin{align*}
\frac{d}{du}\overline{\Delta_j(\log u)} &= \frac{1}{u}\iint_U -\overline{s}u^{-\overline{s}}g_j(s)d\sigma dt = \frac{1}{u}\overline{\Delta'_j(\log u)}\\
&\ll\frac{x^{2m}}{u}|\Delta_j(x^{(j)}_0)|+O(x^{2m}e^{-Ax}).
\end{align*}

Hence, using partial summation and \eqref{eq:Estimate}, gives
\begin{align*}
&\sum_{1\leq k\ne l\leq m}{\sum_p}^* a_{L_k}(p)\overline{a_{L_l}(p)}\Delta_k(\log p)\overline{\Delta_l(\log p)}\\
&\qquad\qquad=\sum_{1\leq l\ne k\leq m}\int_{X_1}^{X_2}\Delta_{k}(\log u)\overline{\Delta_{l}(\log u)}d\phi_{k,l}(u)\\
&\qquad\qquad\ll \frac{e^x}{x^{2m+2}}\sum_{1\leq k\ne l\leq m}|\Delta_k(x^{(k)}_0)||\Delta_l(x^{(l)}_0)|+O(e^{(-A+1-\sigma_1)x})\\
&\qquad\qquad\quad+\sum_{1\leq k\ne l\leq m}\int_{X_1}^{X_2}\frac{u}{(\log u)^{2m+2}}\left|\left(\Delta_k(\log u)\overline{\Delta_l(\log u)}\right)'\right|du\\
&\qquad\qquad\ll \frac{e^x}{x^{2m+2}}\sum_{1\leq j\leq m}|\Delta_j(x^{(j)}_0)|^2+O(e^{(-A+1-\sigma_1)x})\\
&\qquad\qquad\quad+x^{2m}\sum_{1\leq j\leq m}|\Delta_j(x^{(j)}_0)|^2 \int_{X_1}^{X_2}\frac{1}{(\log u)^{2m+2}}du\\
&\qquad\qquad\ll \frac{e^x}{x^{2m+2}}\sum_{1\leq j\leq m}|\Delta_j(x^{(j)}_0)|^2+O(e^{(-A+1-\sigma_1)x}):=E(x)
\end{align*}
where $X_1=e^{x'}$, $X_2=e^{x'+\frac{B_m}{x'^{2m}}}$.

Therefore, by \eqref{eq:SelbergSA}, we get
\begin{align*}
S(x) &= {\sum_p}^* \sum_{j=1}^m|a_{L_j}(p)|^2|\Delta_j(\log p)|^2 + E(x)\\
&\gg \sum_{j=1}^m\left(|\Delta_j(x^{(j)}_0)|^2 + |\Delta_j(x^{(j)}_0)|O(e^{-Ax}) + O(e^{-2Ax})\right){\sum_p}^*|a_{L_j}(p)|^2 + E(x)\\
&\gg \frac{e^x}{x^{2m+1}} \sum_{j=1}^m |\Delta_j(x^{(j)}_0)|^2 + O(e^{(-A+1-\sigma_1)x})+E(x)\\
&\gg \frac{e^x}{x^{2m+1}} \sum_{j=1}^m |\Delta_j(x^{(j)}_0)|^2 + O(e^{(-A+1-\sigma_1)x})+O\left(\frac{e^x}{x^{2m+2}}\sum_{j=1}^m|\Delta_j(x^{(j)}_0)|^2\right)
\\
&\gg \frac{e^x}{x^{2m+1}} \left(\sum_{j=1}^m |\Delta_j(x^{(j)}_0)|\right)^2 + O(e^{(-A+1-\sigma_1)x})
\\
&\gg \frac{e^{(1-\sigma_2)x}}{x^{2m+1}} \sum_{j=1}^m |\Delta_j(x^{(j)}_0)| + O(e^{(-A+1-\sigma_1)x}).
\end{align*}

On the other hand, since $a_{L_j}(p)\ll p^\varepsilon$ for every $\varepsilon >0$, we have
\begin{align*}
S(x)&\ll e^{\varepsilon x}{\sum_p}^* \left|\sum_{j=1}^m a_{L_j}(p)\Delta_j(\log p)\right|\sum_{j=1}^m |\Delta_j(\log p)|\\
&\ll e^{\varepsilon x}{\sum_p}^* \left|\sum_{j=1}^m a_{L_j}(p)\Delta_j(\log p)\right|\sum_{j=1}^m |\Delta_j(x^{(j)}_0)|+O(e^{(-A+1+\varepsilon-\sigma_1)x}).
\end{align*}
Finally, dividing the last inequalities by $\sum_{j=1}^m |\Delta_j(x^{(j)}_0)|$ and taking sufficiently large $A>0$ gives
\[
{\sum_p}^* \left|\sum_{j=1}^m a_{L_j}(p)\Delta_j(\log p)\right|\gg \frac{e^{x(1-\sigma_2-\varepsilon)}}{x^{2m+1}}
\]
and the proof is complete.\end{proof}

\section{Proof of Theorem \ref{th:main}}

Now we shall use the denseness lemma proved above, to give the proof of joint universality for a collection of $L$-functions $L_1,\ldots,L_m$ from the Selberg class. In order to do it we need a joint limit theorem for the following probabilistic measure on $(H(D)^m,\mathcal{B}(H(D)^m))$, where $\mathcal{B}(H(D)^m)$ denotes the class of Borel sets of $H(D)^m$. Basically, the proof of the joint limit theorem and the remaining steps of the proof of Theorem \ref{th:main} are based on  \cite[Chapter 12]{S}, where Steuding proved conditional joint universality (see \cite[Theorem 12.5]{S}) for a slightly different class of $L$-functions. The modification needed are easy and straightforward. Nevertheless, we give a sketch of the proof for sake of completeness.

For 
\[
\underline{L}(s) = (L_1(s),\ldots,L_m(s))
\]
define a probabilistic measure $\mathbf{P}_T^{\underline{L}}$ by
\[
\mathbf{P}_T^{\underline{L}}(A) = \meas{\underline{L}(s+i\tau)\in A},\qquad \text{for}\ A\in\mathcal{B}(H(D)^m).
\]

Moreover, it is known that 
\[
\underline{L}(s,\omega) :=(L_1(s,\omega),\ldots,L_m(s,\omega)),\qquad(\omega\in\Omega),
\]
is an $H(D)^m$-valued random element on $(\Omega,\mathcal{B}(\Omega),m_H)$. Therefore, denoting the distribution of $\underline{L}(s,\omega)$ by $\mathbf{P}^{\underline{L}}$ on $(H(D)^m,\mathcal{B}(H(D)^m))$, gives the following joint limit theorem.
\begin{theorem}[{\cite[Theorem 12.1]{S}}]\label{th:mainLimit}
For $L_1,\ldots,L_m\in \Sclass$ the probability measure $\mathbf{P}_T^{\underline{L}}$ converges weakly to $\mathbf{P}^{\underline{L}}$, as $T\to\infty$.
\end{theorem}

The immediate consequence of the above theorem is the following result. 
\begin{cor}
Let $L_1,\ldots,L_m\in\Sclass$ and $D_M:=\{s\in\mathbb{C}:\sigma_1<\Re(s)<\sigma_2,\ |t|<M\}$ for any $M>0$. Then the probability measure
\[
\mathbf{Q}_T^{\underline{L}}(A):=\meas{\underline{L}(s+i\tau)\in A},
\]
for $A\in \mathcal{B}(H(D_M)^m)$, converges weakly, as $T\to\infty$, to 
\[
\mathbf{Q}^{\underline{L}}(A):=m_H\left\{\omega\in\Omega:\underline{L}(s,\omega)\in A\right\}
\]
for $A\in \mathcal{B}(H(D_M)^m)$.
\end{cor} 

Hence, in order to prove Theorem \ref{th:main} it remains to determine the support of the measure $\mathbf{Q}_T^{\underline{L}}$, which is implied by Hurwitz's classical result on zeros of uniformly convergent sequence of functions. Let us recall that the support of the probabilistic space $(S,\mathcal{B}(S),\mathbf{P})$ is the minimal closed set with measure $1$. It means that the support consists of all elements $x\in S$ satisfying $\mathbf{P}(V)>0$ for every neighborhood $V$ of $x$. By using \eqref{eq:Dense}, \cite[Lemma 12.7]{S} and the definition of support, and modifying the proof of \cite[Lemma 12.6]{S}, we have the following lemma.

\begin{lemma} 
The support of the measure $\mathbf{Q}_T^{\underline{L}}$ is the set
\[
\mathbf{S}_M := \{\underline{\varphi}:=(\varphi_1,\ldots,\varphi_m)\in H(D_M)^m: \underline{\varphi(s)}\ne 0\text{ for }s\in D_M,\text{ or }\underline{\varphi}\equiv 0\}.
\]
\end{lemma}

Now, we are ready to complete the proof of Theorem \ref{th:main}.

\begin{proof}[Proof of Theorem \ref{th:main}]
By Mergelyan's approximation theorem it suffices (see the proof of \cite[Theorem 12.5]{S}) to assume that $g_1,\ldots,g_m$ have non-vanishing analytic continuation to $D_M$, where $M>0$ is such that $K_1,\ldots,K_m\subset D_M$. Then, by the last lemma, $(g_1,\ldots, g_m)$ is an element of the support $\mathbf{S}_M$. Therefore, using the fact that $\mathbf{Q}_T^{\underline{L}}$ converges weakly to $\mathbf{Q}^{\underline{L}}$ and the fact that the set $\Phi$ of functions $\underline{\varphi}\in H(D_M)^m$ satisfying 
\[
\max_{1\leq j\leq m}\max_{s\in K_j}|\varphi_j(s)-g_j(s)|<\varepsilon
\]
is open, yields
\begin{multline*}
\liminf_{T\to\infty}\meas{\max_{1\leq j\leq m}\max_{s\in K_j}|L_j(s+i\tau)-g_j(s)|<\varepsilon}\\
=\liminf_{T\to\infty}\mathbf{Q}_T^{\underline{L}}(\Phi)\geq \mathbf{Q}^{\underline{L}}(\Phi)>0,
\end{multline*}
which completes the proof.
\end{proof}

\section{Examples}\label{sec:Ex}

In this section we give examples of $L$-functions from analytic number theory satisfying Selberg's conjecture, and, particularly, the assumptions of Theorem \ref{th:main}. 

Let us start with a general discussion about joint universality of the Riemann zeta function $\zeta(s)$ and $L$-function $L(s)$ from the Selberg class. In this case, it suffices to assume that $L(s)$ satisfies $\eqref{eq:SelbergSA}$ and 
\begin{equation*}
\sum_{p\leq x} a_L(p) \ll \frac{x}{(\log x)^A},\qquad\text{for arbitrary $A>0$}.
\end{equation*}
It is well known, that there is a strong relation between the error term in the above estimation and zero-free region of $L(s)$. For example, \cite[Theorem 5.13]{IK} states that the prime number theorem for general $L$-function holds under the assumption of existence of the zero-free region. More precisely, one can deduce that for any function $1\ne L\in\Sclass$ with polynomial Euler product we have
\[
\sum_{p\leq x} a_L(p) \ll m_L \li x+ O\left(xe^{-c'\sqrt{\log x}}\right)\qquad\text{for some $c'>0$},
\]
provided there exists $c>0$ such that 
\begin{equation}\label{eq:zeroFree}
L(\sigma+it)\ne 0\qquad\text{for \ $\sigma> 1-\frac{c}{\log(|t|+2)}$, \ $t\in\mathbb{R}$,}
\end{equation} 
except a real zero $\beta<1$. Therefore, we can easily deduce joint universality of the Riemann zeta function $\zeta(s)$ and any entire $L$-function from the Selberg class  with zero-free region of the form \eqref{eq:zeroFree}. It means that, for example, we can show $\zeta(s)$ and any Hecke $L$-function $L_{\mathbf{K}}(s;\chi)$ associated to a finite extension $\mathbf{K}$ of $\mathbb{Q}$ and a non-principle primitive gr\"ossencharacker $\chi$ are jointly universal in the strip $\sigma_{\textup{m}}(L_{\mathbf{K}}(s;\chi))<\sigma<1$. Similarly, we can show that the Riemann zeta function and Artin $L$-function associated to a finite Galois extension are jointly universal. The last example of this kind can be delivered by the theory of classical automorphic $L$-functions. For instance, the normalized $L$-function $L(s,f)$ associated to holomorphic primitive cusp form. Here, we refer to Iwaniec and Kowalski \cite[Chapter 5]{IK} for the proofs of needed prime number theorems for Hecke, Artin and automorphic $L$-functions and more examples of $L$-functions jointly universal with the Riemann zeta function.

Next, consider the joint universality property for $\zeta(s)$ and $L$-function $L(s)$ with a pole at $s=1$ of order $m_L$ satisfying $0<m_L<d_L$. Then it turns out that instead of \eqref{eq:SelbergSB} it suffices to assume the truth of Selberg's conjecture $\eqref{eq:SelbergB}$ with $R(x)\ll 1$ and the existence of a zero-free region for $L(s)$. Indeed, it is well known  (see \cite{CG} or \cite[Theorem 2.4.1]{FHIK}) that every function in $\Sclass$ can be factored into primitive elements. Let us  recall that $F\in\Sclass$ is primitive if $F=F_1F_2$ for $F_1,F_2\in\Sclass$ implies $F_1=1$ or $F_2=1$. Furthermore, Selberg's conjecture \eqref{eq:SelbergB} with $R(x)\ll 1$ implies that the Riemann zeta function is the only primitive element of $\Sclass$ with a pole (see \cite{CG} or \cite[Theorem 2.5.2]{FHIK}). More precisely, under Selberg's Conjeture \ref{conj:Selberg}, every given function $L\in\Sclass$ with a pole at $s=1$ of order $m_L$ can be factored into $m_L$-th power of $\zeta(s)$ and an entire function from $\Sclass$. Therefore, assuming \eqref{eq:zeroFree} for a given $L\in\Sclass$ with $0<m_L<d_L$ and recalling again \cite[Theorem 5.13]{IK} gives that we can factor $L(s)$ into $\zeta(s)^{m_L}$ and an entire function $1\ne L^*(s)\in\Sclass$, which, obviously, has no zeros at least in the same region as $L(s)$ and satisfies \eqref{eq:SelbergSB}. Moreover, $L^*$ satisfies \eqref{eq:SelbergSA} as $L$ does, since one can easily observe that Selberg's conjecture \eqref{eq:SelbergB} with $R(x)\ll 1$ gives
\begin{align*}
m_l\log\log x + O(1)&= \sum_{p\leq x}\frac{|a_L(p)|^2}{p} = \sum_{p\leq x}\frac{|a_{\zeta^{m_L}}(p)+a_{L^*}(p)|^2}{p}\\
&=\sum_{p\leq x}\frac{|a_{\zeta^{m_L}}(p)|^2}{p} + \sum_{p\leq x}\frac{|a_{L^*}(p)|^2}{p} + O(1)\\
&=m_L\log\log x + \sum_{p\leq x}\frac{|a_{L^*}(p)|^2}{p} + O(1).
\end{align*}
Since, additionally, $L^*$ is entire, we can show, by the previous reasoning, that $\zeta(s)$ and $L^*(s)$ are jointly universal in the strip $\sigma_{\textup{m}}(L^*)<\sigma<1$. Thus, it is easy to see that $\zeta(s)$ and $L(s)$ are jointly universal in the same strip, provided $L(s)$ satisfies \eqref{eq:SelbergSA}, \eqref{eq:zeroFree} and  Selberg's conjecture \eqref{eq:SelbergB} holds for every $L$-function with $R(x)\ll 1$.

As an example of application of this observation, we can consider Dedekind zeta function $\zeta_{\mathbf{K}}(s)$ associated to any algebraic number field $\mathbf{K}$. Then it is known that \eqref{eq:zeroFree}, \eqref{eq:SelbergSA} and \eqref{eq:SelbergSB} hold for any algebraic number field (cf. \cite[Section 5.10]{IK}). Hence, $\zeta_{\mathbf{K}}$ can be written as $\zeta(s)L^*(s)$, which implies the joint universality theorem for $\zeta(s)$ and $\zeta_{\mathbf{K}}(s)$ in the strip $\sigma_{\textup m}(L^*)<\sigma<1$ under the assumption of Selberg's orthonormality conjecture.

Let us note that usually the abscissa of the mean-square is smaller for $L$-functions of smaller degree $d_L$, namely \cite[Corollary 6.11]{S} says that $\sigma_{\textup m}(L)<\max(\frac{1}{2},1-\frac{1}{d_L})$. Therefore, the above approach by factorization of $L$-function usually gives universality for a wider strip than the direct proof of joint universality for given $L$-functions.  For example, following \cite[Section 2]{M} let us consider a normalized holomorphic Hecke eigen cusp form $f$, the automorphic $L$-function $L(s,f)$ and the symmetric square $L$-function $L(s,\textup{sym}^2f)$ (for the definition see \cite[Eq.  (2.4) and (2.6)]{M}). It is known that the Rankin-Selberg $L$-function $L(s,f\otimes g)$ is a function of degree $4$ and it is universal (see \cite{Ma} and \cite{Na}) in the strip $3/4<\sigma<1$. However, one can easily show that
\[
L(s,f\otimes f) = \zeta(s)L(s,\textup{sym}^2f).
\]
and it is known that the abscissa of the mean-square of $L(s,\textup{sym}^2f)$ is at most $2/3$. Therefore, using \cite[Eq. (3.8)]{M}, we obtain joint universality for $\zeta(s)$ and $L(s,\textup{sym}^2f)$ in the strip. It implies joint universality for the Riemann zeta function and the automorphic $L$-function in the wider strip $2/3<\sigma<1$. 

It turns out that the theory of the Rankin-Selberg convolution delivers more examples for application of our Theorem \ref{th:main}. It is known that the Rankin-Selberg convolution and the Rankin-Selberg square are  powerful tools to investigate the existence of prime number theorem for automorphic $L$-functions. For example, Iwaniec and Kowalski \cite[Section 5]{IK} showed that the existence of the Rankin-Selberg $L$-function $L(s,f\otimes g)$ implies the existence of its zero-free region, provided some additional conditions related to automorphic forms $f$, $g$ hold. Moreover, they proved that zero-free region for automorphic $L$-function gives prime number theorem (see \cite[Theorem 5.13]{IK}). Note that the coefficients $\lambda_{f\otimes g}(p)$ of the Rankin-Selberg convolution $L(s,f\otimes g)$ satisfy $\lambda_{f\otimes g}(p) = \lambda_f(p)\lambda_g(p)$, where $\lambda_f(p)$ and $\lambda_g(p)$ are coefficients of automorphic $L$-functions $L(s,f)$ and $L(s,g)$ associated to automorphic forms $f$ and $g$, respectively. In particular, the coefficients of the Rankin-Selberg square $L(s,f\otimes\overline{f})$ satisfy $\lambda_{f\otimes \overline{f}}(p) = |\lambda_f(p)|^2$. Therefore, we obtain that the existence of the Rankin-Selberg convolution and the Rankin-Selberg square implies the strong version of Selberg's conjecture, namely
\begin{align}
\sum_{p\leq x}|\lambda_f(p)|^2 &= \kappa_f \li x + O\left(xe^{-c\sqrt{\log x}}\right)\qquad (\kappa_f>0),\label{eq:autoA}\\
\sum_{p\leq x}\lambda_f(p)\overline{\lambda_g(p)} &= O\left(xe^{-c\sqrt{\log x}}\right)\qquad (f\ne g).\label{eq:autoB}
\end{align}

The existence of the Rankin-Selberg convolution and square as well as zero-free region are well investigated for many automorphic $L$-functions. For example, it is known (see \cite[Theorem 5.41]{IK}) that $L(s,f\otimes g)$ has no zero in the region  \eqref{eq:zeroFree} except possibly a one simple zero $\beta<1$, provided $f$ and $g$ are classical primitive modular forms. Hence, we get that \eqref{eq:autoA} and \eqref{eq:autoB} hold and we get joint universality for any collection of automorphic $L$-function $L(s,f_1),\ldots,L(s,f_m)$ with distinct classical primitive modular forms, provided they belong to $\Sclass$. 

Similarly, the result of Liu and Ye \cite[Theorem 2.3]{LY} implies joint universality for a quite general automorphic $L$-functions $L(s,\pi_j)$, $j=1,2,\ldots,m$, associated to irreducible unitary cuspidal representation $\pi_j$ of $GL_m(\mathbb{Q}_A)$ satisfying $\pi_i\not\cong\pi_j\otimes |\det|^{i\tau}$ for any $\tau\in\mathbb{R}$, provided they are elements of the Selberg class. 

It should be noted that, most likely, the Selberg class consists only of automorphic $L$-functions in which case it is widely believed and known for many examples that instead of Selberg's Conjecture \ref{conj:Selberg} we can expect \eqref{eq:autoA} and \eqref{eq:autoB}. It means that probably there is no example of $L$-functions from Selberg class satisfying Selberg's Conjecture \ref{conj:Selberg}, which do not fulfill \eqref{eq:SelbergSA} and \eqref{eq:SelbergSB}. Thus, we conjecture that we do not loss of generality by assuming the stronger version of Selberg's conjecture.

\end{document}